\documentclass[a4paper,10pt]{article}
\usepackage[latin1]{inputenc}
\usepackage[T1]{fontenc}
\usepackage{amssymb}
\usepackage{epsfig}
\usepackage{amssymb}
\usepackage{amsmath}
\usepackage{amsfonts}
\usepackage{graphicx, url}
\usepackage{url}
\usepackage{makeidx}
\usepackage{cite}
\usepackage{fancyhdr}
\usepackage{color}
\usepackage{setspace}
\newtheorem{theorem}{Theorem}[section]
\newtheorem{corollary}[theorem]{Corollary}

\newtheorem{remark}[theorem]{Remark}

\newtheorem{definition}[theorem]{Definition}
\newtheorem{proposition}[theorem]{Proposition}
\newtheorem{example}[theorem]{Example}

\newenvironment {proof} {{\it Proof.}}{\hspace*{\fill}$\Box$\par\vspace{4mm}}
\usepackage{amsmath}
\usepackage{amsfonts}

\newcommand{\mc}{\mathcal}
\newcommand{\mb}{\mathbb}

\title{Singular recursive utility}

\author{K. R. Dahl \thanks{\protect\normalsize Department of Mathematics, University of Oslo, Pb. 1053 Blindern, 0316 Oslo, Norway. Email: kristrd@math.uio.no} \and B. {\O}ksendal \thanks{\protect\normalsize Department of Mathematics, University of Oslo.  This research was carried out with partial support of CAS - Centre for Advanced Study, at the Norwegian Academy of Science and Letters, within the research program SEFE.   Email: oksendal@math.uio.no.}}

\date{16 March 2017\\
\emph{The final version of this paper will be published in Stochastics.}}

\begin{document}

\maketitle

\begin{abstract}
We introduce the concept of singular recursive utility. This leads to a kind of singular BSDE which, to the best of our knowledge, has not been studied before. We show conditions for existence and uniqueness of a solution for this kind of singular BSDE. Furthermore, we analyze the problem of maximizing the singular recursive utility. We derive sufficient and necessary maximum principles for this problem, and connect it to the Skorohod reflection problem. Finally, we apply our results to a specific cash flow. In this case, we find that the optimal consumption rate is given by the solution to the corresponding Skorohod reflection problem.
\end{abstract}


\smallskip


\section{Introduction}

Let $c(t) \geq 0$ be a consumption rate process. The classical way of measuring the total utility of $c$ from $t=0$ to $t=T$ is by the expression

\[
 J(c) = E[\int_0^T U(t,c(t))dt]
\]

\noindent where $U(t, \cdot)$ is a utility function for each $t$. This way of adding utility rates over time has been criticized from an economic and modeling point of view. See e.g. Mossin (1969) and Hindy, Huang \& Kreps (1992).

Instead, Duffie and Epstein (1992) proposed to use \emph{recursive utility} $Y(t)$, defined as the solution of the backward stochastic differential equation (BSDE)

\begin{equation}
 \label{eq: Rec_utility_BSDE}
Y(t) =E[\int_t^T g(s, Y(s), c(s))ds | \mc{F}_t]; t \in [0,T].
\end{equation}

Thus, we see that $Y(0) = J(c)$ in the special case where $g(s,y,c) = U(s,c)$ does not depend on $y$.

The question is: How should we model the recursive utility of a \emph{singular} consumption process $\xi$? A possible proposal is

\begin{equation}
 \label{eq: BSDE_rec_utility_singular}
Y(t)= E[\int_t^T g(s, Y(s), \xi(s)) d \xi(s)| \mc{F}_t].
\end{equation}

%If $d\xi(s) = c(s) ds$, we get back to the classical case.

If we write

\[
 Y(t) = E[\int_0^T g(s, Y(s), \xi(s)) d \xi(s) | \mc{F}_t] - \int_0^t g(s, Y(s), \xi(s)) d \xi(s)
\]

\noindent we get by the martingale representation theorem (see for instance {\O}ksendal (2007)) that $Y(t)$ can be written (in the Brownian motion case):

\[
 Y(t) = -\int_0^t g(s, Y(s), \xi(s)) d \xi(s) + \int_0^t Z(s) dB(s)
\]

\noindent for some adapted process $Z$. Thus, $(Y,Z)$ soves the \emph{singular BSDE}

\begin{equation}
 \label{eq: singular_BSDE}
\begin{array}{lll}
dY(t) &=& -g(t,Y(t),\xi(t))d\xi(t) + Z(t)dB(t) \\[\smallskipamount]
Y(T) &=& 0.
\end{array}
\end{equation}

To the best of our knowledge, such singular BSDEs have not been studied before. Wang (2004) studies a somewhat similar problem. However, the BSDE in Wang (2004) is, in contrast to the BSDE~(\ref{eq: singular_BSDE}), a mix of singular- and Lebesgue integration. Another related paper is Diehl and Friz (2012) which studies BSDEs with rough drivers. Although the BSDEs studied in Diehl and Friz (2012) resembles equation~\eqref{eq: singular_BSDE}, our function $g$ is more general than the corresponding function in Diehl and Friz (2012).

Other recent extensions to the literature on recursive utility include Kraft and Seifried (2014) which derives stochastic differential utility as a limit of resursive utility and Belak et al. (2016) which derives a theory for backward nonlinear expectation equations and defines recursive utility in this framework. Marinacci and Montrucchio (2010) study uniqueness of solutions for stochastic recursive utilities. In addition, Kraft et al.(2017) studies a consumption and investment problem with recursive utility.

The aim of this paper is to study singular BSDEs such as~\eqref{eq: singular_BSDE} and to maximize the corresponding singular recursive utility. In Section~\ref{sec: LinSingularBSDE} the singular BSDE is analyzed. In Sections~\ref{sec: max_sing_rec_utility}-\ref{sec: NecMaxPrinc}, sufficient and necessary maximum principles for the singular recursive utility problem are derived. Finally, we apply these results in Section~\ref{sec: applications} to solve an optimal consumption problem for a specific cash flow. The resulting optimal consumption rate is the solution of a Skorohod problem, which (if the Skorohod problem has a solution) is a local time. Hence, the optimal consumption rate is in general singular. This is in contrast to the classical solution of the optimal consumption problem, which only permits consumption processes which are absolutely continuous.

\section{Problem formulation}

Consider a probability space $(\Omega, \mc{F}, P)$. In this space, we let $B(t)$ be a Brownian motion and $\tilde{N}(dt, \cdot)=N(dt,\cdot)-\nu(\cdot)dt$ be an independent compensated Poisson random measure. We assume that the L\'evy measure $\nu$ of the Poisson random measure $N$ satisfies
$$ \int_{\mathbb{R}} \zeta^2 \nu(d\zeta) < \infty.$$
On the other hand, we allow that, for all $\epsilon > 0$,
$$ \int_0^{\epsilon}\zeta \nu(d\zeta) = \infty,$$
so $\nu$ may have infinite variation on any time interval $[0,\epsilon].$\\
We let $\mathbb{F}=\{ \mc{F}_t \} _{t \in [0,T]}$ be the natural filtration generated by the Brownian motion and the compensated Poisson random measure.

Also, consider a consumption exposed cash flow $X(t)= X^{\xi}(t)$ modeled by a stochastic differential equation (SDE) with jumps as follows:

\begin{equation}
 \label{eq: SDE_X}
 \begin{array}{llll}
dX(t) &=& b(t, X(t))dt + \sigma(t, X(t))dB(t) + \int_{\mb{R}} \beta(t, X(t^-), \zeta) \tilde{N}(dt, d\zeta) -d\xi(t) \\[\smallskipamount]
X(0) &=& x
\end{array}
\end{equation}

\noindent (we suppress the $\omega$ for notational simplicity). \\
Here, $\xi(t) := \xi(t,\omega)$ is the stochastic consumption process, assumed to be cadlag and non-decreasing and satisfying $\xi(0) = 0$. This control $\xi$ is to be chosen from a set of admissible controls, $\mc{A}$. We let $\mc{A}$ be the set of all finite variation stochastic processes $\xi$ which are adapted, c{\`a}dl{\`a}g and with increasing components as well as satisfying $\xi(0^-) = 0$\\

Associated to $\xi$ we introduce a \emph{singular recursive utility process} $Y(t)$ represented by a \emph{singular backward stochastic differential equation (BSDE)} with jumps, as follows. \\

\begin{definition}
Let $g(t,y,\xi,\omega): [0,T] \times \mb{R} \times \mb{R} \times \Omega \rightarrow \mb{R}$ be a given predictable function, Lipschitz wrt. $y$ and $\xi$ and uniformly continuous wrt. $t$, called the \emph{driver}. Also, let $h(x,\omega): \mb{R} \times \Omega \rightarrow \mb{R}$ be a given bounded $\mathcal{F}_T$-measurable random variable for each $x$, called a \emph{terminal time payoff function}. Then we define the \emph{singular recursive utility process} $Y(t)=Y^{\xi}(t)$ with respect to $\xi$ as the first component of the solution $(Y(t), Z(t), K(t,\cdot))$ of the following singular backward stochastic differential equation:
%Assume that the triple
%$$(Y(t)=Y^{(\xi)}(t), Z(t) = Z^{(\xi)}(t), K(t,\zeta)=K^{(\xi)}(t,\zeta))$$
%is the solution of the following singular BSDE,

\begin{equation}
\begin{array}{lll}
\label{eq: BSDE_Y}
 dY(t) &=& -g(t,Y(t),\xi(t))d\xi(t) + Z(t)dB(t) + \int_{\mb{R}} K(t^-, \zeta) \tilde{N}(dt, d\zeta) \mbox{ for } t \leq T, \\[\smallskipamount]
 Y(T) &=& h(X^{\xi}(T)).
\end{array}
\end{equation}

\end{definition}

%\noindent Note that equation~(\ref{eq: BSDE_Y}) is a BSDE in the unknowns $(Y_g^{(\zeta)}(t), Z(t), K(t, \cdot))$.

For more on (non-singular) BSDEs with jumps, see for instance {\O}ksendal and Sulem (2007).

We would like to solve the following optimal consumption problem with respect to singular recursive utility, for a given \emph{driver} $g$ and a given terminal payoff function $h$:
\vskip 0.5cm

$${\bf{PROBLEM}}$$

Find $\xi^* \in \mc{A}$ such that
\begin{equation}
 \label{eq: opt_cont_problem}
\sup_{\xi \in \mc{A}} Y^{\xi}(0) = Y^{\xi^*}(0).
\end{equation}

In other words, we would like to maximize the singular recursive utility of the control $\xi$.

\section{Singular BSDE with drift term}

Let $g: \mb{R}_+ \times \mb{R} \times \mb{R} \times \Omega \rightarrow \mb{R}$ and $b : \mb{R}_+ \times \mb{R} \times \Omega$ be functions which are bounded for $t \in [0,T]$. Consider the following singular BSDE with drift:

\begin{equation}
 \label{eq: singular_BSDE_2igjen_NY}
\begin{array}{lll}
dY(t) &=& g(t,Y(t), \xi(t))d\xi(t) + b(t, Y(t))dt - Z(t)dB(t) \\[\smallskipamount]
&& \quad \quad \quad -  \int_{\mb{R}} K(t^-, \zeta) \tilde{N}(dt, d\zeta) \mbox{ for } t \leq T \\[\smallskipamount]
 Y(T) &=& X_{\xi}.
\end{array}
\end{equation}
\noindent where $\xi$ is a given singular process and $X_{\xi}$ is a given $\mc{F}_T$-measurable random variable, which may depend on $\xi$, such that $E[|X_{\xi}|] < \infty$.

%In order to prove existence of solution to this equation, we need the following lemma.

%\begin{lemma}
%\label{lemma: 1}
%For two general functions $f_1, f_2$,

%\[
% E[|E_t[\int_t^T f_1(s) - f_2(s) d\xi(s) ]|^2] \leq K(t) E[\int_t^T |f_1(s) - f_2(s)|^2]] d\xi(s)
%\]

%\noindent where $K(t) := E[\xi(T) - \xi(t)] \geq 0$.
%\end{lemma}

%\begin{proof}
%\begin{equation}
%\label{eq: regning_Holder}
%\begin{array}{lll}
% E[|E_t[\int_t^T f_1(s) - f_2(s) d\xi(s) ]|^2] &\leq& E[ E_t[\int_t^T |f_1(s) - f_2(s)| d\xi(s)]^2] \\[\smallskipamount]
% &\leq& E[ E_t[(\int_t^T |f_1(s) - f_2(s)| d\xi(s))^2]] \\[\smallskipamount]
% &=& E[ (\int_t^T |f_1(s) - f_2(s)| d\xi(s))^2] \\[\smallskipamount]
% &\leq& E[\int_t^T 1^2 d\xi(s)] E[\int_t^T |f_1(s) - f_2(s)|^2 d\xi(s)] \\[\smallskipamount]
% &=& K(t) E[\int_t^T |f_1(s) - f_2(s)|^2] d\xi(s)]
%\end{array}
% \end{equation}

%\noindent where $K(t) := E[\xi(T) - \xi(t)]$, $E_t[\cdot]$ denotes the conditional expectation wrt. $\mc{F}_t$. The second inequality follows from Jensen's inequality for conditional expectation, the first equality from the rule of double expectation and the final inequality from H{\"o}lder's inequality. Also, $K(t) \geq 0$ for all $t$ since $\xi$ is increasing.

%\end{proof}

\begin{theorem}
\label{thm: existence_uniqueness_sing_BSDE}
(Existence and uniqueness of solution to the singular BSDE with drift) Assume the following Lipschitz-type assumption: There exists constants $C_1, C_2 > 0$ such that for any two stochastic processes $Y_1, Y_2$,
\[
 E[\int_t^T |g(s,Y_1(s), \xi(s)) - g(s, Y_2(s), \xi(s))| d\xi(s)] \leq C_1  E[\int_t^T E[|Y_1(s) - Y_2(s)|] d\xi(s)]
\]

\noindent and

\[
 E[\int_t^T |b(s,Y_1(s)) - b(s, Y_2(s))| ds] \leq C_2  \int_t^T E[|Y_1(s) - Y_2(s)|] ds.
\]

Also, assume that $E[X_{\xi} + \int_0^T g(s, 0, \xi(s)) d \xi(s) + \int_0^T b(s, Y_s) ds | \mc{F}_t]$ is square integrable and that $\xi$ is continuous. Then, there exists a unique solution $(Y,Z,K)$ of the singular BSDE with drift~\eqref{eq: singular_BSDE_2igjen_NY}.
\end{theorem}

\begin{proof}
Define a sequence $\{Y_t^n\}_n$, $n \geq 0$, by $Y_t^0 = 0$ and

\[
 Y_t^{n+1} = E[X_{\xi} + \int_t^T g(s, Y_s^n, \xi(s)) d\xi(s) + \int_t^T b(s, Y_s^n) ds| \mc{F}_t].
\]

Then,

\begin{equation}
 \label{eq: calculations_use_again_NY}
\begin{array}{llll}
 \varphi_{n+1}(t) &:=& E[|Y_t^{n+1} - Y_t^{n}|] \\[\smallskipamount]
&=& E[|E[\int_t^T \{g(s,Y_s^n, \xi(s)) - g(s, Y_s^{n-1}, \xi(s)) \} d \xi(s) \\[\smallskipamount]
&&+ \int_t^T \{b(s,Y_s^n) - b(s, Y_s^{n-1}) ds\} | \mc{F}_t]|] \\[\smallskipamount]
&\leq& E[E[\int_t^T |g(s,Y_s^n, \xi(s)) - g(s, Y_s^{n-1}, \xi(s))| d \xi(s) \\[\smallskipamount]
&&+\int_t^T |b(s,Y_s^n) - b(s, Y_s^{n-1})| ds| \mc{F}_t]]
  \\[\smallskipamount]
&=& E[\int_t^T |g(s,Y_s^n, \xi(s)) - g(s, Y_s^{n-1}, \xi(s))| d \xi(s) \\[\smallskipamount]
&&+\int_t^T |b(s,Y_s^n) - b(s, Y_s^{n-1})| ds]
  \\[\smallskipamount]
&\leq& C_1 E[\int_t^T E[|Y_s^n - Y_s^{n-1}|] d \xi(s)] + C_2 \int_t^T E[|Y_s^n - Y_s^{n-1}|] ds \\[\smallskipamount]
&=& C_1 E[\int_t^T \varphi_n(s) d \xi(s)] + C_2 \int_t^T \varphi_n(s) ds
\end{array}
\end{equation}

\noindent where the third equality follows from the rule of double expectation, the first inequality follows from the Minkowski inequality and that $\xi$ has increasing components and the second inequality from the Lipschitz assumptions.

Then, by iterating the previous inequality, we find that for all $t$

\[
\begin{array}{llll}
\varphi_{n+1}(t) &\leq& \sup_{t \in [0,T]} \varphi_1(t) \sum_{m=0}^n C_1^{n-m} C_2^{m} \frac{E[\xi(T)]^{n-m} T^m}{n!} \\[\medskipamount]
&=& \frac{\sup_{t \in [0,T]} \varphi_1(t)}{n!} (C_1 E[\xi(T)])^n \sum_{m=0}^n (\frac{C_2 T}{C_1 E[\xi(T)]})^m \\[\medskipamount]
&=& \frac{\sup_{t \in [0,T]} \varphi_1(t)}{n!} (C_1 E[\xi(T)])^n \frac{1 - (\frac{C_2 T}{C_1 E[\xi(T)]})^{n+1}}{1 - \frac{C_2 T}{C_1 E[\xi(T)]}}
\end{array}
\]
\noindent by summing the finite geometric series. Here, the inequality for the iterated $d\xi$-integrals follows from the It{\^o} lemma for semimartingales, see e.g. Protter~\cite{Protter}, Theorem 3.2. This means that $Y^n$, $n \geq 1$, is a Cauchy sequence (since factorials grow faster than exponentials). Let $\hat{Y} := \lim_{n \rightarrow \infty} Y^n$.

Now, let $n \rightarrow \infty$ in the definition of $Y^n_t$. Then,

\[
 \hat{Y}_t = \lim_{n \rightarrow \infty} Y^n_t =E[X_{\xi} + \int_t^T g(s, \hat{Y}_s, \xi(s)) d\xi(s) + \int_t^T b(s, \hat{Y}_s) ds | \mc{F}_t].
\]

We would like to show that $\hat{Y}_t$ has a right continuous version which is the solution of the singular BSDE. Let $M_t$ be the right continuous version of the martingale $E[X_{\xi} + \int_0^T g(s, \hat{Y}_s, \xi(s)) d \xi(s) + \int_0^T b(s, \hat{Y}_s) ds| \mc{F}_t]$. Let

\begin{equation}
 \label{eq: y_lik_yhatt_NY}
\begin{array}{lll}
 Y_t &:=& M_t - \int_0^t g(s, \hat{Y}_s, \xi(s)) d\xi(s) - \int_t^T b(s, \hat{Y}_s) ds \\[\smallskipamount]
 &=& E[X_{\xi} + \int_t^T g(s,\hat{Y}_s, \xi(s)) d \xi(s) + \int_t^T b(s, \hat{Y}_s) ds| \mc{F}_t] = \hat{Y}_t \quad \mbox{  $P$-a.s.}
\end{array}
\end{equation}

\noindent and $Y_t$ is right continuous. Then,

\[
 \begin{array}{lll}
Y_t &=& X_{\xi} + M_t - M_T + \int_0^T g(s, \hat{Y}_s, \xi(s)) d\xi(s) - \int_0^t g(s, \hat{Y}_s, \xi(s)) d \xi(s) \\[\smallskipamount]
&& +  \int_0^T b(s, \hat{Y}_s) ds -  \int_0^t b(s, \hat{Y}_s) ds \\[\smallskipamount]
&=& X_{\xi} + \int_t^T g(s, \hat{Y}_s, \xi(s)) d\xi(s) + \int_t^T b(s, \hat{Y}_s) ds - (M_T-M_t) \\[\smallskipamount]
&=& X_{\xi} + \int_t^T g(s, Y_s, \xi(s)) d\xi(s) + \int_t^T b(s, Y_s) ds - (M_T-M_t)
 \end{array}
\]

\noindent where the final equality follows from equation~\eqref{eq: y_lik_yhatt_NY}.

$M_t$ is a martingale and from the assumptions it is square integrable, so  the martingale representation theorem (see {\O}ksendal and Sulem (2007)) implies that there exists processes $Z_t, K_t(\cdot)$ such that

\[
\begin{array}{lll}
 M_t &=& E[M_t] + \int_0^t Z_s dB_s + \int_0^t \int_{\mb{R}}  K_s(\zeta) \tilde{N}(ds, d\zeta) \\[\smallskipamount]
 &=& E[M_0]+ \int_0^t Z_s dB_s + \int_0^t \int_{\mb{R}}  K_s(\zeta) \tilde{N}(ds, d\zeta).
\end{array}
\]

So,

\[
 M_T - M_t = \int_t^T Z_s dB_s + \int_t^T \int_{\mb{R}} K_s(\zeta) \tilde{N}(dt, d\zeta).
\]

Hence,

\[
 \begin{array}{lll}
  dY_t &=& g(t, Y_t) d \xi(t) + b(t, Y(t))dt - Z_t dB_t - \int_{\mb{R}} K(t, \zeta) \tilde{N}(dt, d\zeta),\\[\smallskipamount]
  Y_T &=& X_{\xi}.
 \end{array}
\]

This means that $(Y_t,Z_t, K_t(\cdot))$ solves the singular BSDE~\eqref{eq: singular_BSDE_2igjen_NY}.

We can also prove uniqueness of solution of the singular BSDE with drift: Consider the singular BSDE~\eqref{eq: singular_BSDE_2igjen_NY}. We would like to show that this equation has a unique solution. Let $(Y_1(t), Z_1(t), K_1(t, \cdot))$ and $(Y_2(t), Z_2(t), K_2(t, \cdot))$ be two solutions of equation~\eqref{eq: singular_BSDE_2igjen_NY}. Define

\[
   v(t):= E[|Y_1(t) - Y_2(t)|].                                                                                                                                                                                                                                                                      \]

Then, by the same kind of calculations as in \eqref{eq: calculations_use_again_NY},

\begin{equation}
\label{eq: v_NY}
\begin{array}{lll}
 v(t) &=& E[|Y_1(t) - Y_2(t)|] \\[\smallskipamount]
 &\leq& C_1 E[\int_t^T v(s) d\xi(s)] + C_2 \int_t^T v(s) ds
\end{array}
 \end{equation}
\noindent where we have used the Minkowski inequality, as well as the Lipschitz assumptions.
%Note that this implies that $v(T) =0$ (since $v(t) \geq 0$ for all $t \in [0,T]$).

The inequality~(\ref{eq: v_NY}) implies that

\[
v(t) \leq 2 \max\{C_1, C_2\} \max\{E[\int_t^T v(s) d\xi(s)], \int_t^T v(s) ds\}.
\]

Therefore, by using either the regular or the stochastic (backward) Gr{\"o}nwall inequality (see Lemma 2.1 in Ding and Wu (1998) and Lemma 4.7 in Cohen and Elliot~\cite{Cohen}) depending on the value of the final maximum above, this implies that $v(t) \leq 0$. However, by the definition of $v(t)$, this again implies that $v(t)=0$. Hence, the solution of the singular BSDE~\eqref{eq: singular_BSDE_2igjen_NY} is unique.

%EV SINGUAER GRONWALL HER I STEDEN.

%Now, consider a stochastic process $w(t)$ satisfying the following BSDE:
%
%\[
% \begin{array}{lll}
%  dw(t) &=& \tilde{C} X(t) w(t) d\xi(t) \\[\smallskipamount]
% w(T) &=& 0,
% \end{array}
%\]
%
%Then, $w(t) =0$ for all $t \in [0,T]$ (due to the terminal condition) \textcolor{red}{ER DETTE RIKTIG?}. Also, because of the computations~\eqref{eq: v}, $v(t) \leq w(t)$ for all $t \in [0,T]$ \textcolor{red}{ER DETTE RIKTIG?}. So, $0 \leq v(t) \leq 0$ for all $t \in [0,T]$. Therefore, $v(t) =0$ for all $t \in [0,T]$, which implies that $Y_1=Y_2$. Since $Y_1 = Y_2$, then it must also hold that $Z_1=Z_2$ and $K_1(\cdot) = K_2(\cdot)$ (from the stucture of the singular BSDE~\eqref{eq: singular_BSDE_2igjen}), and hence the solution of the singular BSDE is unique.

\end{proof}

\section{The linear singular BSDE}
\label{sec: LinSingularBSDE}

Let $\phi, \alpha, \beta, c : \mb{R}_+ \times \Omega \rightarrow \mb{R}$  and $\gamma : \mb{R}_+ \times \mb{R} \times \Omega \rightarrow \mb{R}$ be functions. We would like to solve the singular BSDE in the linear case, i.e. we want to solve

\begin{equation}
\label{eq: lin_sing_BSDE}
 \begin{array}{lll}
  dY_t &=& -g(t, Y_t) d \xi(t) + Z_t dB_t + \int_{\mb{R}} K(t, \zeta) \tilde{N}(dt, d\zeta),\\[\smallskipamount]
  Y_T &=& X_{\xi}.
 \end{array}
\end{equation}

\noindent when

\[
\begin{array}{lll}
g(t, Y(t), \xi(t)) &=& \phi(t) + \alpha(t)Y(t) + c(t)\xi(t).
\end{array}
\]

We have the following theorem:

\begin{theorem}
\label{thm: lin_sing_BSDE}
Assume that $\xi$ is continuous. Then,
\begin{equation}
\label{eq: guess_Y}
Y(t)=E[\frac{\Gamma(T)}{\Gamma(t)}X + \int_t^T \frac{\Gamma(s)}{\Gamma(t)} \big(\phi(s) + c(s)\xi(s)\big) d\xi(s) |\mc{F}_t]
\end{equation}

\noindent is the unique solution of the linear singular BSDE~\eqref{eq: lin_sing_BSDE}, where

\begin{equation}
\label{eq: guess_gamma}
\begin{array}{lll}
d\Gamma(t) &=& \Gamma(t^-)\alpha(t)d\xi(t) \\[\smallskipamount]
\Gamma(0)&=&1.
\end{array}
\end{equation}

\end{theorem}

\begin{proof}
We use the same method as in the proof of Theorem 1.7 in {\O}ksendal and Sulem (2007). By It{\^o}'s formula,

\[
\begin{array}{lll}
d(\Gamma(t)Y(t)) &=& \Gamma(t^-)dY(t)+ Y(t^-)d\Gamma(t)+ d[\Gamma, Y](t) \\[\smallskipamount]

&=& \Gamma(t^-) \Big[-\big( \phi(t)+ \alpha(t)Y(t) + c(t)\xi(t) \big) d\xi(t) + Z(t)dB(t) \\[\smallskipamount]
&&+ \int_{\mb{R}} K(t,\zeta) \tilde{N}(dt, d\zeta) \Big] + Y(t^-)\Gamma(t^-) \Big[\alpha(t)d\xi(t) + \beta(t)dB(t) \\[\smallskipamount]
&&+ \int_{\mb{R}} \gamma(t,\zeta) \tilde{N}(dt, d\zeta) \Big] \\[\smallskipamount]

&=& \Gamma(t^-) \Big[-\big( \phi(t) + c(t)\xi(t) \big) d\xi(t) + Z(t)dB(t) \\[\smallskipamount]
&&+ \int_{\mb{R}} K(t,\zeta) \tilde{N}(dt, d\zeta) + Y(t^-) \big(\beta(t)dB(t) \\[\smallskipamount]
&&+ \int_{\mb{R}} \gamma(t,\zeta) \tilde{N}(dt, d\zeta) \big) \Big]
\end{array}
\]

So $\Gamma(t)Y(t) + \int_0^t \Gamma(t) \big( \phi(t) + c(t)\xi(t) \big) d\xi(t)$ is a martingale. Hence,

\[
\begin{array}{lll}
\Gamma(t)Y(t) + \int_0^t \Gamma(s) \big( \phi(s) + c(s)\xi(s) \big) d\xi(s) \\[\smallskipamount]
\hspace{1.5cm}= E \Big[ \Gamma(T)Y(T) + \int_0^T \Gamma(s) \big( \phi(s) + c(s)\xi(s) \big) d\xi(s) | \mc{F}_t \Big].
\end{array}
\]

Therefore,

\[
\Gamma(t)Y(t) = E \Big[ \Gamma(T)X_{\xi} + \int_t^T \Gamma(s) \big( \phi(s) + c(s)\xi(s) \big) d\xi(s) | \mc{F}_t \Big]
\]
\noindent which is the claim of the theorem.
\end{proof}

%Note: only N on final line above due to Oksendal, Sulem ex. 1.7.

%If I understood you correctly, the final quadratic covariation, i.e. $d(\xi \xi)(t)$ term does not need to be $0$, but is in general the sum of the jumps of the process $\xi$. We would like to cancel lots of terms on the right hand side above, so that the right hand side is something which is not a function of $Y, Z, K$ + a martingale. However, because we have a $d\xi$-integral instead of a $dt$ integral, it is not possible to cancel the terms necessary. For a completely general process $\Gamma$ (not necessarily the one previously defined), we have:

%\[
%\begin{array}{lll}
%d(\Gamma(t)Y(t)) &=& \Gamma(t^-)dY(t) Y(t^-)d\Gamma(t) d(\Gamma Y)(t) \\[\smallskipamount]
%&=& \Gamma(t^-)[-(\phi(t)+ \alpha(t)Y(t) + \beta(t)Z(t) \\[\smallskipamount]
%&&+ \int_{\mb{R}} \gamma(t, \zeta) K(t,\zeta) \nu(d\zeta) + c(t)\xi(t)) d\xi(t) + Z(t)dB(t) \\[\smallskipamount]
%&&+ \int_{\mb{R}} K(t,\zeta) \tilde{N}(dt, d\zeta)] + Y(t^-) d\Gamma(t^-)
% + Z(t) \mbox{($dB$-term from $\Gamma$)}dt \\[\smallskipamount]
%&&+ \int_{\mb{R}} K(t, \zeta) \gamma(t,\zeta) \mbox{($d\tilde{N}$-term from %$\Gamma$)} N(dt,d\zeta) \\[\smallskipamount]
%&&+ \sum_{\mbox{jumps of $\xi$}} \{-g(t, Y, Z, K, \xi) \mbox{ ($d\xi$-term from $\Gamma$)} \}(\Delta \xi)^2
%\end{array}
%\]

%I cannot see how it is possible to cancel terms here in order to conclude as in {\O}ksendal and Sulem~\cite{OksendalSulem}, Theorem 1.7, regardless of how $\Gamma$ is chosen.

\section{Maximizing singular recursive utility}
\label{sec: max_sing_rec_utility}

In the following, let $\mb{F}=\{\mc{F}_t\}_{t \in [0,T]}$ be the filtration generated by (only) the Brownian motion. Consider the following forward stochastic differential equation (FSDE):

\begin{equation}
 \label{eq: notat_(1)}
 \begin{array}{lll}
 dX(t) &=& b(t, X(t))dt + \sigma(t, X(t))dB(t) + \theta(t,X(t)) d \xi(t); \\[\smallskipamount]
 X(0) &=& x \in \mb{R}.
 \end{array}
\end{equation}
\noindent where we assume that the functions $b, \theta, \sigma$ are differentiable w.r.t. $x$ (the second component) with bounded derivatives for $t \in [0,T]$. 

Also, consider the singular backward stochastic differential equation (SBSDE):

\begin{equation}
 \label{eq: notat_(2)}
 \begin{array}{lll}
 dY(t) &=&-g_1(t,X(t),Y(t),Z(t),\xi(t))dt -g_2(t,Y(t), \xi(t))d \xi(t) + Z(t)dB(t); \\[\smallskipamount]
 Y(T) &=& h(X(T))
 \end{array}
\end{equation}
\noindent where $h : \mb{R} \rightarrow \mb{R}$ is a given concave, $C^1$ function, differentiable with bounded derivative on $[0,T]$. Also, the functions $g_1, g_2$ are assumed to be bounded for $t \in [0,T]$, differentiable wrt. $x,y,z$ and $y$ respectively with bounded partial derivatives (for $t \in [0,T]$).

The \emph{singular recursive utility functional} is defined by

\begin{equation}
 \label{eq: notat_(3)}
 \begin{array}{lll}
  J(\xi) = E[\int_0^T f(t, X(t)) dt + \varphi(X(T))] + \psi(Y(0))
 \end{array}
\end{equation}
\noindent where $\varphi$ and $\psi$ are given concave, $C^1$ functions, and the function $f$ is bounded for $t \in [0,T]$ as well as differentiable wrt. x with a bounded partial derivative.

The problem is to find a $\xi^* \in \mc{A}$ and $\Phi$ such that

\begin{equation}
 \label{eq: notat_(4)}
 \Phi := \sup_{\xi \in \mc{A}} J(\xi) = J(\xi^*).
\end{equation}

This is a \emph{singular forward-backward SDE (SFBSDE) control problem}. To the best of our knowledge this type of problem has not been studied before. Some related papers are the following: In {\O}ksendal and Sulem (2014) and {\O}ksendal and Sulem (2015), maximum principles for optimal control of \emph{non-singular} FBSDEs are established. In {\O}ksendal and Sulem (2012) and Hu et al. (2014), maximum principles for singular control are proved, but the results do not apply to the singular recursive utility case given in equation~\eqref{eq: notat_(2)}.

Our present paper is combining ideas from these and related papers to establish maximum principles for optimal control of a coupled system of FBSDEs.

To this end, define the Hamiltonian $\mathcal{H}$ by

%\begin{align}
% \label{eq: notat_(5)}
%H(t,x,y,\xi,p,q,\lambda)(dt,d\xi(t)) &= H_0(t,x,y,p,q)dt + H_1(t,y,\xi,p,\lambda)d \xi(t)
%&+H_2(t,y,\xi,\Delta p,\Delta \lambda) \Delta \xi(t)
% \end{align}
%
%\noindent where

\begin{align}
 \label{eq: notat_(6)}
\mathcal{H}(t,x,y,z,\xi,p,q,\lambda)(dt,d\xi)
=H_1(t,x,y,z,\xi,p,q,\lambda) dt + H_2(t,x,y,\xi)d\xi(t)
\end{align}
where
\begin{align}
&H_1(t,x,y,z,\xi,p,q,\lambda)= f(t,x) + b(t,x)p +  \sigma(t,x)q+
 \lambda g_1(t,x,y,z,\xi)\\
&H_2(t,x,y,\xi,p,\lambda) = p\theta(t,x) + \lambda g_2(t,y,\xi).
%&H_2(t,y,\xi,\Delta p,\Delta \lambda)=  \theta(t) \Delta p + g(t,y,\xi)\Delta\lambda
 \end{align}

The equations for the adjoint variables, which are a kind of generalized Lagrange multipliers, $p(t), q(t), \lambda(t)$ are:
\begin{itemize}
\item
BSDE for $p(t), q(t)$:

\begin{equation}
 \label{eq: notat_(7)}
 \begin{array}{llll}
  dp(t) &=& -\frac{\partial H_1(t)}{\partial x}(t)dt -\frac{\partial H_2(t)}{\partial x}(t)d\xi(t)+ q(t)dB(t); \mbox{ } 0 \leq t \leq T \\[\smallskipamount]
  p(T) &=& \varphi'(X(T)) + \lambda(T)h'(X(T)).
 \end{array}
\end{equation}

\item
FSDE for $\lambda(t)$:

\begin{equation}
 \label{eq: notat_(8)}
 \begin{array}{llll}
  d\lambda(t) &=& \frac{\partial H_1}{\partial y}(t)dt+ \frac{\partial H_2}{\partial y}(t)d\xi(t)+\frac{\partial H_1}{\partial z}(t)dB(t)
  %+ \frac{\partial H_1}{\partial y}(t)d\xi(t)
  %+\frac{\partial H_2}{\partial y}(t)\Delta\xi(t)
   ; \mbox{ } 0 \leq t \leq T \\[\smallskipamount]
  \lambda(0) &=& \psi'(Y(0)).
 \end{array}
\end{equation}
\end{itemize}

In the general case, existence and uniqueness of solutions of systems for singular FBSDEs such as~\eqref{eq: notat_(7)}-\eqref{eq: notat_(8)} is not known, see also Remark~\ref{remark: 5.2}.

Then, the following maximum principle holds:

\begin{theorem}
\label{eq: Suff_Max_Princ}
(Sufficient maximum principle for optimal singular recursive utility control)

Let $\hat{\xi} \in \mc{A}$, with associated solutions $\hat{X}(t), \hat{Y}(t), \hat{Z}(t), \hat{p}(t), \hat{q}(t), \hat{\lambda}(t)$ of the coupled FBSDE system \eqref{eq: notat_(1)}-\eqref{eq: notat_(2)} and \eqref{eq: notat_(7)}-\eqref{eq: notat_(8)}. Assume the following:
\begin{itemize}
\item{Continuity:}
$\hat{\xi}(t)$ is continuous
\item{Concavity:}
The functions $\varphi$ and $ \psi$ are $C^1$ and concave, $\psi' \geq 0$, and for each $t$ the map
\begin{align}
 \label{eq: notat_(9)}
 &(x,y,z,\xi) \mapsto
  \mathcal{H}(t,x,y,z,\xi,\hat{p}(t), \hat{q}(t),\hat{\lambda}(t))(dt,d\xi(t))
\end{align}
is concave.
\item{Variational inequality:}
\begin{align}
 \label{eq: notat_(10)}
 &\sup_{\xi} \mathcal{H}(t, \hat{X}(t),\hat{Y}(t), \hat{Z}(t),\xi, \hat{p}(t), \hat{q}(t),\hat{\lambda}(t))(dt,d\xi)\nonumber\\
 & = \mathcal{H}(t, \hat{X}(t),\hat{Y}(t), \hat{Z}(t),\hat{\xi}(t), \hat{p}(t), \hat{q}(t),\hat{\lambda}(t))(dt,d\hat{\xi}(t))
\end{align}
i.e.
 \begin{align}
 \label{eq: notat_10_2}
 & \hat{\lambda}(t)g_1(t,\hat{X}(t),\hat{Y}(t),\hat{Z}(t),\xi)dt+\big( \hat{p}(t)\theta(t,\hat{X}(t))+\hat{\lambda}(t)g_2(t,\hat{Y}(t),\xi(t))\big) d\xi(t) \nonumber\\
 \leq & \hat{\lambda}(t)g_1(t,\hat{X}(t),\hat{Y}(t),\hat{Z}(t),\hat{\xi})dt+\big( \hat{p}(t)\theta(t,\hat{X}(t))+\hat{\lambda}(t)g_2(t,\hat{Y}(t),\hat{\xi}(t))\big) d\hat{\xi}(t)
 \end{align}
 for all $\xi$ (where the differential inequalities means that the corresponding inequalities hold when integrated).
\end{itemize}
Then $\hat{\xi}$ is an optimal control for the problem~\eqref{eq: notat_(4)}.

\end{theorem}

\begin{proof}
Choose $\xi \in \mc{A}$ and consider, with $\hat{X}(t) = X^{\hat{\xi}}(t)$ etc.

 \begin{equation}
  \label{eq: notat_(11)}
  J(\xi) - J(\hat{\xi}) = J_1 + J_2 + J_3,
 \end{equation}
\noindent where

\[
 \begin{array}{lll}
  J_1 &=& E[\int_0^T \{f(t) - \hat{f}(t)\} dt] \mbox{ where } \hat{f}(t) := f(t, \hat{X}(t)), \\[\smallskipamount]

  J_2 &=& E[\varphi(X(T)) - \varphi(\hat{X}(T))] \\[\smallskipamount]

  J_3 &=& \psi(Y(0)) - \psi(\hat{Y}(0)); \mbox{ where } Y(0) =  Y^{\xi}(0), \hat{Y}(0) = Y^{\hat{\xi}}(0).
 \end{array}
\]

By the definition of $H_1$ we have, with $\tilde{b} := b-\hat{b}$ etc.,

\begin{equation}
 \label{eq: notat_(12)}
 \begin{array}{llll}
  J_1 =  E[\int_0^T \big( H_1(t) - \hat{H}_1(t)
  %- \hat{\lambda}(t)\{g(t)- \hat{g}(t)\} d \hat{\xi}(t) \\[\smallskipamount]
 - \hat{p}(t)\tilde{b}(t) - \hat{q}(t) \tilde{\sigma}(t) -\hat{\lambda}(t) \tilde{g}_1(t))
 %-\hat{\lambda}_2(t) g_2(t)\big)
 dt]
 %- \hat{p}(t)\theta(t) \big( d \xi(t) - d \hat{\xi}(t)\big )%[\smallskipamount]
% &&-\Delta \hat{\lambda}(t)\{ g(t)\Delta \xi(t)-\hat{g}(t) \Delta \hat{\xi}(t)\}.
 %+[\hat{\theta}(t)\Delta \hat{p}(t) + \hat{g}(t)\Delta \hat{\lambda}(t) \big) ] \Delta \hat{\xi}(t)
 \end{array}
\end{equation}
where we  have used the shorthand notation
\begin{align}
&H_1(t)=H_1(t,X(t),Y(t),Z(t),\xi(t),\hat{p}(t),\hat{q}(t),\hat{\lambda}(t)), \nonumber\\
&\hat{H}_1(t)=H_1(t,\hat{X}(t),\hat{Y}(t),\hat{Z}(t),\hat{\xi}(t),\hat{p}(t),\hat{q}(t),\hat{\lambda}(t)). \\
%&\text{ and similarly we will write }\nonumber\\
%&H_1(t)=H_1(t,X(t),Y(t),\xi(t),\hat{p}(t),\hat{q}(t),\hat{\lambda}(t)), \nonumber\\
%&= H_0(t,x,y,p,q)dt + H_1(t,y,\xi,p,\lambda)d \xi(t)
%&\hat{H}_1(t) = H_1(t,\hat{X}(t),\hat{Y}(t),\hat{\xi}(t),\hat{p}(t),\hat{q}(t),\hat{\lambda}(t)).
% &= H_0(t,x,y,p,q)dt + H_1(t,y,\xi,p,\lambda)d \xi(t)
\end{align}
By concavity of $\varphi$ and the It{\^o} formula, we have

\begin{equation}
 \label{eq: notat_(13)_1}
 \begin{array}{lll}
 J_2 &\leq& E[\varphi'(\hat{X}(T))\tilde{X}(T)] = E[\{\hat{p}(T) - \hat{\lambda}(T)h'(\hat{X}(T))\} \tilde{X}(T)] \\[\smallskipamount]

 &=& E[\int_0^T \hat{p}(t) \big( \tilde{b}(t)dt + \tilde{\sigma}(t)dB(t) + \theta(t) d \xi(t) -\hat{\theta}(t) d\hat{\xi}(t) \big) \\[\smallskipamount]
 && + \int_0^T \tilde{X}(t) \big( -\frac{\partial \hat{H}_1}{\partial x}(t) dt -\frac{\partial \hat{H}_2}{\partial x}(t) d\xi(t)+ \hat{q}(t)dB(t) \big) + \int_0^T \hat{q}(t) \tilde{\sigma}(t) dt ] \\[\smallskipamount]
 &&- E[\hat{\lambda}(T)h'(\hat{X}(T))\tilde{X}(T)]. \\[\smallskipamount]
 \end{array}
 \end{equation}

 Consider an increasing sequende of stopping times $\tau_n$ defined by

 \[
  \tau_n := T \wedge \inf \{t > 0: \int_0^t \big[(\hat{p}(s)\tilde{\sigma}(s))^2 + (\tilde{X}(s) \hat{q}(s))^2 \big] ds \geq n\}.
 \]

Note that the sequence $\{\tau_n\}_{n=1}^{\infty}$ conveges to $T$ as $n \rightarrow \infty$. Since It{\^o} integrals with $L^2$ integrands have expectation zero, it follows from \eqref{eq: notat_(13)_1} that

\begin{equation}
 \label{eq: localize}
\begin{array}{llll}
 E[\varphi'(\hat{X}(\tau_n))\tilde{X}(\tau_n)] &=& E[\int_0^{\tau_n} \hat{p}(t) \{ \tilde{b}(t)dt + \theta(t) d \xi(t) - \hat{\theta}(t)d\hat{\xi}(t) \} \\[\smallskipamount]
 && - \int_0^{\tau_n} \tilde{X}(t)  \frac{\partial \hat{H}_1}{\partial x}(t) dt - \int_0^{\tau_n} \tilde{X}(t)  \frac{\partial \hat{H}_2}{\partial x}(t) d\xi(t)+ \int_0^{\tau_n} \hat{q}(t) \tilde{\sigma}(t) dt ] \\[\smallskipamount]
 &&- E[\hat{\lambda}(\tau_n)h'(\hat{X}(\tau_n))\tilde{X}(\tau_n)].
\end{array}
\end{equation}

By passing to the limit in \eqref{eq: localize}, and using the dominated convergence theorem (which can be applied due to our assumptions on the coefficient functions), we find that

\begin{equation}
 \label{eq: localize_2}
 \begin{array}{lll}
 E[\varphi'(\hat{X}(T))\tilde{X}(T)] &=&  E[\int_0^T \hat{p}(t) \{ \tilde{b}(t)dt + \theta(t) d \xi(t) - \hat{\theta}(t)d\hat{\xi}(t) \} \\[\smallskipamount]
 && - \int_0^T \tilde{X}(t)  \frac{\partial \hat{H}_1}{\partial x}(t) dt - \int_0^T \tilde{X}(t)  \frac{\partial \hat{H}_2}{\partial x}(t) d\xi(t)+ \int_0^T \hat{q}(t) \tilde{\sigma}(t) dt ] \\[\smallskipamount]
 &&- E[\hat{\lambda}(T)h'(\hat{X}(T))\tilde{X}(T)].
 \end{array}
 \end{equation}

By combining \eqref{eq: notat_(13)_1} and \eqref{eq: localize_2}, we find that

\begin{equation}
 \label{eq: notat_(13)}
 \begin{array}{lll}
 J_2 &\leq& E[\int_0^T \hat{p}(t) \{ \tilde{b}(t)dt + \theta(t) d \xi(t) - \hat{\theta}(t)d\hat{\xi}(t) \} \\[\smallskipamount]
 && - \int_0^T \tilde{X}(t)  \frac{\partial \hat{H}_1}{\partial x}(t) dt - \int_0^T \tilde{X}(t)  \frac{\partial \hat{H}_2}{\partial x}(t) d\xi(t)+ \int_0^T \hat{q}(t) \tilde{\sigma}(t) dt ] \\[\smallskipamount]
 &&- E[\hat{\lambda}(T)h'(\hat{X}(T))\tilde{X}(T)].
 \end{array}
\end{equation}

By concavity of $\psi$ and the It{\^o} product rule we get,

\begin{equation}
 \label{eq: notat_(14)_1}
 \begin{array}{lll}
  J_3 &=& \psi(Y(0)) - \psi(\hat{Y}(0)) \leq \psi'(\hat{Y}(0)) \tilde{Y}(0) \\[\smallskipamount]

  &=& \hat{\lambda}(0) \tilde{Y}(0) = E[\hat{\lambda}(T)\tilde{Y}(T)] \\[\smallskipamount]
  && - E[\int_0^T \tilde{Y}(t) d \hat{\lambda}(t) + \int_0^T \hat{\lambda}(t) d \tilde{Y}(t)+\int_0^T \frac{\partial \hat{H}_1}{\partial z}(t)\tilde{Z}(t)dt]. \\[\smallskipamount]

 \end{array}
\end{equation}

Again, consider an increasing sequence of stopping times $\bar{\tau}_n$ defined by

\[
 \bar{\tau}_n := T \wedge \inf \{t > 0: \int_0^t \big[(\tilde{Y}(s)\frac{\partial \hat{H}_1}{\partial z}(s))^2 + (\hat{\lambda}(s)\tilde{Z}(s))^2 \big] ds \geq n\}.
\]

Similarly as before, the sequence $\{\bar{\tau}_n\}_{n=1}^{\infty}$ conveges to $T$ as $n \rightarrow \infty$. Since It{\^o} integrals with $L^2$ integrands have expectation zero, it follows from \eqref{eq: notat_(14)_1}, the It{\^o} formula and the concavity of $h$ that

\begin{equation}
\label{eq: notat_(14)_2}
 \begin{array}{lll}
  J_3 &\leq& -E[\int_0^{\bar{\tau}_n}  \frac{\partial \hat{H}_1}{\partial y}(t) \tilde{Y}(t)dt+ \frac{\partial \hat{H}_1}{\partial z}(t)\tilde{Z}(t)dt +\frac{\partial \hat{H}_2}{\partial y}(t) \tilde{Y}(t)d\hat{\xi}(t))] \\[\smallskipamount]

  &- &E[\int_0^{\bar{\tau}_n} \hat{\lambda}(t)\big(\tilde{g}_1(t)dt+ g_2(t) d \xi(t) - \hat{g}_2(t) d \hat{\xi}(t)\big)] \\[\smallskipamount]

  &&  + E[\hat{\lambda}(\bar{\tau}_n) h'(\hat{X}(\bar{\tau}_n)) \tilde{X}(\bar{\tau}_n)].

 \end{array}
\end{equation}

By letting $n \rightarrow \infty$ in \eqref{eq: notat_(14)_2} and using the dominated convergence theorem (this can be applied due to our assumptions on the coefficient functions, which implies that $E[|\tilde{Y}|], E[|\tilde{Z}|] < \infty$ by the proof of Theorem~\ref{thm: existence_uniqueness_sing_BSDE}), we see that

 \begin{equation}
\begin{array}{llll}
\label{eq: notat_(14)}
  J_3 &\leq& -E[\int_0^{T}  \frac{\partial \hat{H}_1}{\partial y}(t) \tilde{Y}(t)dt+ \frac{\partial \hat{H}_1}{\partial z}(t)\tilde{Z}(t)dt +\frac{\partial \hat{H}_2}{\partial y}(t) \tilde{Y}(t)d\hat{\xi}(t))] \\[\smallskipamount]

  &- &E[\int_0^{T} \hat{\lambda}(t)\big(\tilde{g}_1(t)dt+ g_2(t) d \xi(t) - \hat{g}_2(t) d \hat{\xi}(t)\big)] \\[\smallskipamount]

  &&  + E[\hat{\lambda}(T) h'(\hat{X}(T) \tilde{X}(T)].
\end{array}
\end{equation}

Adding \eqref{eq: notat_(12)}, \eqref{eq: notat_(13)} and \eqref{eq: notat_(14)} we get, by concavity of $\mathcal{H}$,

\begin{equation}
 \begin{array}{llll}
 & J(\xi) - J(\hat{\xi}) \nonumber\\
 &\leq E[\int_0^T \{H_1(t) - \hat{H}_1(t) - \frac{\partial \hat{H}_1}{\partial x}(t) \tilde{X}(t) - \frac{\partial \hat{H}_1}{\partial y}(t) \tilde{Y}(t)- \frac{\partial \hat{H}_1}{\partial z}(t) \tilde{Z}(t)\}dt\nonumber\\
  &+\hat{p}(t)\big( \theta(t)d\xi(t)-\hat{\theta}(t)d\hat{\xi}(t)\big)-\frac{\partial \hat{H}_2}{\partial x}(t) \tilde{X}(t)d\hat{\xi}(t) -\frac{\partial \hat{H}_2}{\partial y}(t) \tilde{Y}(t)d\hat{\xi}(t) \nonumber\\
  &+\hat{\lambda}(t)\big(g_2(t,Y(t),\xi(t))d\xi(t)-g_2(t,\hat{Y}(t),\hat{\xi}(t))d\hat{\xi}(t)\big)
  %-\frac{\partial \hat{H}}{\partial y}(t) \tilde{Y}(t)(dt+d\hat{\xi}(t))]
  \nonumber\\
 &=E[\int_0^T \{\mathcal{H}(t) - \hat{\mathcal{H}}(t) - \frac{\partial \hat{\mathcal{H}}}{\partial x}(t) \tilde{X}(t) - \frac{\partial \hat{\mathcal{H}}}{\partial y}(t) \tilde{Y}(t)- \frac{\partial \hat{\mathcal{H}}}{\partial z}(t) \tilde{Z}(t)\}\nonumber\\
 % &+\{ H_1(t) d\xi(t)- \hat{H}_1(t)d\hat{\xi}(t)
 % -\frac{\partial \hat{H}}{\partial y}(t) \tilde{Y}(t)(dt+d\hat{\xi}(t))\}]\nonumber\\
 &\leq E[\int_0^T \langle \nabla_{\xi} \hat{\mathcal{H}}(t), \xi(\cdot) - \hat{\xi}(\cdot)\rangle ] \leq 0,

 \end{array}
\end{equation}

\noindent since $\xi = \hat{\xi}$ maximizes $\mathcal{H}$.
Here
\begin{align}
&\langle \nabla_{\xi} \hat{\mathcal{H}}(t), \xi(\cdot) - \hat{\xi}(\cdot)\rangle\nonumber\\
&=\hat{\lambda}(t)[ \frac{\partial \hat{g}_1}{\partial \xi}(t) \big( \xi(t)-\hat{\xi}(t)\big) dt + \frac{\partial \hat{g}_2}{\partial \xi}(t) \big( \xi(t)-\hat{\xi}(t)\big)d\hat{\xi}(t)]+\hat{H}_2(t)(d\xi(t)-d\hat{\xi}(t))
\end{align}
is the action of the gradient (Fr\' echet derivative)  $\nabla_{\xi} \hat{\mathcal{H}}(t)$ of $\mathcal{H}$ on $\xi(\cdot)-\hat{\xi}(\cdot)$, i.e. the directional derivative of $\mathcal{H}$ in the direction $\xi -\hat{\xi}$.
\end{proof}

\begin{remark}
 \label{remark: 5.2}
 Assume that

 \begin{equation}
  \label{eq: 26}
  \begin{array}{lll}
   g_1(t,x,y,x,\xi) &=& g_1(t,x,y,z) \\[\smallskipamount]
   g_2(t,y,\xi) &=& g_2(t,y)
  \end{array}
 \end{equation}
\noindent do not depend on $\xi$.

Then, the variational inequality~\eqref{eq: notat_10_2} is equivalent to the variational inequality

\begin{equation}
 \label{eq: 27}
 \begin{array}{lrlll}
  $(i)$& \mbox{ } \hat{p}(t) \theta(t, \hat{X}(t)) + \hat{\lambda}(t)g_2(t, \hat{Y}(t)) &\leq& 0 \mbox{ for all } t \in [0,T], \\[\smallskipamount]
  $(ii)$& \mbox{ } \{ \hat{p}(t)\theta(t, \hat{X}(t)) + \hat{\lambda}(t)g_2(t,\hat{Y}(t))  \} d \hat{\xi}(t) &=& 0 \mbox{ for all } t \in [0,T].
 \end{array}
\end{equation}

To see this, we first apply \eqref{eq: notat_10_2} to

\[
 d\xi(t) = d\hat{\xi}(t) + d\beta(t)
\]
\noindent where $\beta(t)$ is an increasing continuous adapted process. Then we get

\[
 \{ \hat{p}(t)\theta(t, \hat{X}(t)) + \hat{\lambda}(t)g_2(t,\hat{Y}(t))  \} d \beta(t) \leq 0 \mbox{ for all } t \in [0,T].
\]

Since this holds for all such $\beta$, we deduce that

\begin{equation}
 \label{eq: 28}
 \hat{p}(t)\theta(t, \hat{X}(t)) + \hat{\lambda}(t)g_2(t,\hat{Y}(t)) \leq 0 \mbox{ for all } t \in [0,T].
\end{equation}

On the other hand, if we apply \eqref{eq: notat_10_2} to

\[
 d\xi(t) = \frac{1}{2} d\hat{\xi}(t)
\]
\noindent we get

\begin{equation}
 \label{eq: 29}
 \{ \hat{p}(t)\theta(t, \hat{X}(t)) + \hat{\lambda}(t)g_2(t,\hat{Y}(t))  \} d \hat{\xi}(t) \geq 0 \mbox{ for all } t \in [0,T].
\end{equation}

By combining \eqref{eq: 28} and \eqref{eq: 29} we obtain \eqref{eq: 27}. In particular, \eqref{eq: 27} implies that $\hat{\xi}(t)$ only increases when $\hat{p}(t)\theta(t, \hat{X}(t)) + \hat{\lambda}(t)g_2(t,\hat{Y}(t)) = 0$, and that the corresponding solution $(\hat{X}(t), \hat{Y}(t), \hat{\lambda}(t), \hat{p}(t))$ of the coupled system \eqref{eq: notat_(1)}-\eqref{eq: notat_(2)} and \eqref{eq: notat_(7)}-\eqref{eq: notat_(8)} of forward-backward singular SDEs is \emph{reflected downwards}  at the boundary $\partial G$ of the region

\begin{equation}
 \label{eq: 30}
 G := \{ (t,x,y,\lambda,p) \in \mb{R}^4 ; p \theta(t,x) + \lambda g_2(t,y) < 0 \}.
\end{equation}

Therefore, we see that the optimal singular control $\hat{\xi}$ appears as the \emph{local time} at $\partial G$ of this reflected process. In the special case with just one, possibly multidimensional, singular SDE of the form \eqref{eq: notat_(1)}, the problem to find a solution $(X(t), \xi(t))$ such that

\begin{equation}
 \label{eq: 31}
 \left\{
 \begin{array}{lll}
 X(t) \in \bar{D} \mbox{ for all } t, \\[\smallskipamount]
 \xi \mbox{ is continuous and } \xi \mbox{ increases only when } X(t) \in \partial D
\end{array}
\right.
 \end{equation}

\noindent for a given domain $D$, is called a \emph{Skorohod reflection problem}. The existence and uniqueness of a solution $(\hat{X}(t), \hat{\xi}(t))$ is this case has been proved under certain conditions on the system \eqref{eq: notat_(1)} and the domain $D$. See e.g. Freidlin (1985) and the references therein.

However, for coupled systems of singular forward-backward SDEs, as our system above, the existence and uniqueness of the solution of the corresponding Skorohod reflection problem is not known, to the best of our knowledge. The study of this question is beyond the scope of this paper.

\end{remark}

\section{A necessary maximum principle for singular recursive utility}
\label{sec: NecMaxPrinc}

We can also prove a necessary maximum principle for the singular recursive utility problem. In order to do this we need some additional notation and assumptions:\\
For $\xi \in \mc{A}$ let $\mc{V}(\xi)$ denote the set of $\mathbb{F}$-adapted processes $\beta$ of finite variation such that there exists $\delta=\delta(\xi) > 0$ satisfying
\begin{equation}\label{eq3.11a}
\xi + a \beta \in \mc{A} \text{ for all  } a \in [0,\delta].
\end{equation}

Assume that for all $\xi \in \mc{A}$ and for all $\beta \in \mc{V}(\xi)$ the following derivative processes exist and belong to $L^2([0,T] \times \Omega)$:
\begin{equation}
\label{eq33}
\begin{array}{lll}
x(t) =& \lim_{a \rightarrow 0^+} \frac{X^{\xi + a \beta} - X^{\xi}}{a} (t) \\
y(t) =& \lim_{a \rightarrow 0^+} \frac{ Y^{\xi + a \beta} - Y^{\xi}}{a} (t) \\
z(t) =& \lim_{a \rightarrow 0^+} \frac{ Z^{\xi + a \beta} - Z^{\xi}}{a} (t).
\end{array}
\end{equation}

\begin{remark}
The existence and $L^2$-features of these derivative processes is a non-trivial issue, and we do not discuss conditions for this in our paper. Here we will just assume that these properties hold. We refer to Pr{\'e}v{\^o}t and R{\"o}ckner (2007) for a study of this issue in a related setting.
\end{remark}

Then, by the FSDE~ \eqref{eq: notat_(1)} and the BSDE~\eqref{eq: notat_(2)},

\[
\begin{array}{lll}
dx(t) &=& \frac{\partial b}{\partial x}(t)x(t)dt + \frac{\partial \sigma}{\partial x}(t)x(t)dB(t) + \frac{\partial \theta}{\partial x}(t)x(t)d\xi(t) + \theta(t)d\beta(t), \\[\smallskipamount]
dy(t) &=& -(\frac{\partial g_1}{\partial x}(t)x(t) + \frac{\partial g_1}{\partial y}(t)y(t) + \frac{\partial g_1}{\partial z}(t) z(t) + \frac{\partial g_1}{\partial \xi}(t) \beta(t))dt \\[\smallskipamount]
&&+ (-\frac{\partial g_2}{\partial y}(t) y(t) -\frac{\partial g_2}{\partial \xi }(t)\beta(t)) d\xi(t) - g_2(t,Y,\xi)d\beta(t) + z(t)dB(t).
\end{array}
\]
\noindent Here, we have used that

\[
\begin{array}{lll}
g_2(t, Y^{\xi + a\beta}, \xi +a\beta) d(\xi + a\beta)(t) \\[\smallskipamount]
\hspace{2.5cm}= g_2(t, Y^{\xi + a\beta}, \xi +a\beta) d\xi(t) + a g_2(t, Y^{\xi + a\beta}, \xi +a\beta)d\beta(t)
\end{array}
\]

\noindent which implies (by the product rule) that

\[
\begin{array}{lll}
\lim_{a \rightarrow 0^+} \frac{g_2(t, Y^{\xi + a\beta}, \xi +a\beta) d(\xi + a\beta)(t) - g_2(t, Y, \xi) d\xi(t)}{a} \\[\smallskipamount]
 \hspace{1.5cm}=(-\frac{\partial g_2}{\partial y}(t)y(t) - \frac{\partial g_2}{\partial \xi}(t)\beta(t) )d\xi(t) - g_2(t,Y,\xi)d\beta(t) \\[\smallskipamount]
 \hspace{1.9cm}-\lim_{a \rightarrow 0^+} \{ a(\frac{\partial g_2}{\partial y}(t)y(t) + \frac{\partial g_2}{\partial \xi}(t)\beta(t)) \} \\[\smallskipamount]
\hspace{1.5cm}= (-\frac{\partial g_2}{\partial y}(t)y(t) - \frac{\partial g_2}{\partial \xi}(t)\beta(t) )d\xi(t) - g_2(t,Y,\xi)d\beta(t).
\end{array}
\]

Note also that $x(0) = 0$ and $y(T) = h'(X(T))x(T)$ from the boundary conditions of equations~\eqref{eq: notat_(1)}-\eqref{eq: notat_(2)}.

Define $\mc{T}$ as the set of times where the process $\xi(t)$ jumps, and $\mc{T}_{\beta}$ as the set of times where both processes $\beta(t)$ and $\xi(t)$ jump.

With this in mind, we are ready to prove the necessary maximum principle.

\begin{proposition}
\label{thm: necessary}
Assume that \eqref{eq33}
%and \eqref{eq: derivert_prosesser}
holds. Then the following are equivalent:

\begin{itemize}
\item{$\lim_{a \rightarrow 0^+} \frac{J(\xi + a \beta) - J(\xi)}{a}  \leq 0$ for all $\beta \in \mathcal{V}(\xi)$.}
\item{$
\begin{array}{lll}

E[\int_0^T H_2(t) d \beta(t)]  + E[\int_0^T \beta(t) \{\lambda(t) \frac{\partial g_2}{\partial \xi}(t) d \hat{\xi}(t) + \lambda(t) \frac{\partial g_1}{\partial \xi}(t) dt  \} ] \\[\smallskipamount]
- E[\sum_{t \in \mc{T}_{\beta}} g_2(t^-) y(t^-) \frac{\partial H_1}{\partial y}(t^-) \Delta \xi(t) \Delta \beta(t)] \\[\smallskipamount]
- E[\sum_{t \in \mc{T}} (\frac{\partial g_2}{\partial y}(t^-)y(t^-) + \frac{\partial g_2}{\partial \xi}(t^-)\beta(t^-) ) \frac{\partial H_1}{\partial y}(t^-) (\Delta\xi(t))^2 ] \leq 0 \mbox{ } \forall \mbox{ } \beta \in \mathcal{V}(\xi).
\end{array}
$}
\end{itemize}
\end{proposition}

\begin{proof}

Similarly as in the proof of Theorem~\ref{eq: Suff_Max_Princ}, by introducing a suitable increasing sequence of stopping times converging to $T$, we see that we may assume that all local martingales appearing in the proof below are martingales.
Note that

\[
\lim_{a \rightarrow 0^+} \frac{J(\xi + a \beta) -J(\xi)}{a}  = I_1 + I_2 + I_3
\]

\noindent where

\[
\begin{array}{llll}
I_1 &=& \lim_{a \rightarrow 0^+} \frac{ E[\int_0^T f(t, X^{\xi + a \beta}(t)) dt] - E[\int_0^T f(t, X^{\xi}(t)) dt] }{a}, \\[\smallskipamount]
I_2 &=& \lim_{a \rightarrow 0^+} \frac{E[\varphi(X^{\xi + a\beta}(T))] -E[\varphi(X^{\xi}(T))]}{a} \\[\smallskipamount]
I_3 &=& \lim_{a \rightarrow 0^+} \frac{\psi(Y^{\xi + a \beta}(0)) - \psi(Y^{\xi}(0)) }{a}.
\end{array}
\]

Then, by changing the order of integration and differentiation,

\[
I_1 = E[\int_0^T \frac{\partial f}{\partial x}(t) x(t) dt].
\]

Also,

\[
I_2=E[\varphi'(X(T))x(T)] = E[(p(T) - \lambda(T) h'(X(T))) x(T)]
\]

\noindent and

\[
I_3 = \psi'(Y(0))y(0)= \lambda(0)y(0).
\]

Furthermore, by It{\^o}'s product rule and the definitions

\[
\begin{array}{lll}
I_2 &=& E[\int_0^T p(t) d x(t) + \int_0^T x(t) dp(t) + \int_0^T d[p,x](t)] -E[ \lambda(T) h'(X(T)) x(T)] \\[\smallskipamount]
&=& E[\int_0^T p(t) (\frac{\partial b}{\partial x}(t) x(t) dt + \frac{\partial \theta}{\partial x}(t) x(t) d\xi(t) + \theta(t, X(T))d\beta(t) )] \\[\smallskipamount]
&&- E[\int_0^T x(t) \{\frac{\partial H_1}{\partial x}(t) dt + \frac{\partial H_2}{\partial x}(t) d\xi(t) \}] + E[\int_0^T q(t) \frac{\partial \sigma}{\partial x}(t) x(t) dt] \\[\smallskipamount]
&& -E[ \lambda(T) h'(X(T)) x(T)] \\[\smallskipamount]
&=& E[\int_0^T x(t) \{-\frac{\partial f}{\partial x}(t)dt - \lambda(t) \frac{\partial g_1}{\partial x}(t) dt  \} ] \\[\smallskipamount]
&&+ E[\int_0^T \{ H_2(t) - \lambda(t) g_2(t) \} d \beta(t)] - E[ \lambda(T) h'(X(T)) x(T)] \\[\smallskipamount]
&=& -I_1 - E[\int_0^T x(t) \{ \lambda(t) \frac{\partial g_1}{\partial x}(t) dt  \} ] \\[\smallskipamount]
&&+ E[\int_0^T \{ H_2(t) - \lambda(t) g_2(t) \} d \beta(t)] -E[ \lambda(T) h'(X(T)) x(T)]. \end{array}
\]

Similarly, we see that from It{\^o}'s product rule and the chain rule,

\[
\begin{array}{lll}
I_3 &=& E[\lambda(T)y(T)] - E[\int_0^T \lambda(t) dy(t) + \int_0^T y(t) d\lambda(t) + \int_0^T d[\lambda, y](t)] \\[\smallskipamount]
&=& E[\lambda(T)h'(X(T)) x(T)] - E[\int_0^T \lambda(t) \{( - \frac{\partial g_2}{\partial \xi}(t) \beta(t)) d\xi(t) \\[\smallskipamount]
&&- g_2(t) d\beta(t) - (\frac{\partial g_1}{\partial x}(t)x(t) + \frac{\partial g_1}{\partial \xi}(t) \beta(t)  )dt \}] \\[\smallskipamount]
&&- E[\sum_{t \in \mc{T}_{\beta}} g_2(t^-) y(t^-) \frac{\partial H_1}{\partial y}(t^-) \Delta \xi(t) \Delta \beta(t)] \\[\smallskipamount]
&&- E[\sum_{t \in \mc{T}} (\frac{\partial g_2}{\partial y}(t^-)y(t^-) + \frac{\partial g_2}{\partial \xi}(t^-)\beta(t^-) ) \frac{\partial H_1}{\partial y}(t^-) (\Delta\xi(t))^2 ]
\end{array}
\]
\noindent where $\Delta \xi(t) :=\xi(t)-\xi(t-)$ and $\Delta \beta(t) :=  \beta(t) -  \beta(t-)$. Then, by the previous calculations,

\[
\begin{array}{lll}
I_1 +I_2 +I_3 &=& E[\int_0^T H_2(t) d \beta(t)] + E[\int_0^T \beta(t) \{\lambda(t) \frac{\partial g_2}{\partial \xi}(t) d \hat{\xi}(t) + \lambda(t) \frac{\partial g_1}{\partial \xi}(t) dt  \} ] \\[\smallskipamount]
&&- E[\sum_{t \in \mc{T}_{\beta}} g_2(t^-) y(t^-) \frac{\partial H_1}{\partial y}(t^-) \Delta \xi(t) \Delta \beta(t)] \\[\smallskipamount]
&&- E[\sum_{t \in \mc{T}} (\frac{\partial g_2}{\partial y}(t^-)y(t^-) + \frac{\partial g_2}{\partial \xi}(t^-)\beta(t^-) ) \frac{\partial H_1}{\partial y}(t^-) (\Delta\xi(t))^2 ].
\end{array}
\]

Hence,

\[
\begin{array}{lll}
\frac{d }{da} J(\xi + a \beta)|_{a=0} &=& E[\int_0^T H_2(t) d \beta(t)] + E[\int_0^T \beta(t) \{\lambda(t) \frac{\partial g_2}{\partial \xi}(t) d \hat{\xi}(t) + \lambda(t) \frac{\partial g_1}{\partial \xi}(t) dt  \} ] \\[\smallskipamount]
&&- E[\sum_{t \in \mc{T}_{\beta}} g_2(t^-) y(t^-) \frac{\partial H_1}{\partial y}(t^-) \Delta \xi(t) \Delta \beta(t)] \\[\smallskipamount]
&&- E[\sum_{t \in \mc{T}} (\frac{\partial g_2}{\partial y}(t^-)y(t^-) + \frac{\partial g_2}{\partial \xi}(t^-)\beta(t^-) ) \frac{\partial H_1}{\partial y}(t^-) (\Delta\xi(t))^2 ].
\end{array}
\]

\end{proof}

If we assume that $\xi(t)=\hat{\xi}(t)$ is a continuous process, we get the following corollary to Proposition~\ref{thm: necessary}:

\begin{corollary}
\label{corollary: rewritten}
Assume that \eqref{eq33} holds and that $\hat{\xi}$ is continuous. Then the following are equivalent:

\begin{itemize}
\item{$\lim_{a \rightarrow 0^+} \frac{J(\hat{\xi} + a \beta) - J(\hat{\xi})}{a}  \leq 0$ for all $\beta \in \mathcal{V}(\hat{\xi})$.}
\item{$
\begin{array}{lll}
E[\int_0^T \hat{H}_2(t) d \beta(t)] +  E[\int_0^T \beta(t) \{\hat{\lambda}(t) \frac{\partial \hat{g}_2}{\partial \xi}(t) d \hat{\xi}(t) + \hat{\lambda}(t) \frac{\partial \hat{g}_1}{\partial \xi}(t) dt  \} ]  \leq 0 \\[\smallskipamount]
\end{array}
$
\noindent for all  $\beta \in \mathcal{V}(\hat{\xi})$.
}
\end{itemize}

This final inequality is also equivalent to

\begin{equation}
\label{eq: necessary_condition}
\begin{array}{lll}
 E[\int_0^T \langle \nabla_{\xi} \hat{\mc{H}}(dt, d \hat{\xi}(t)), \beta(t) \rangle ]  \leq 0 \mbox{ for all} \mbox{ } \beta \in \mathcal{V}(\hat{\xi})
\end{array}
\end{equation}
\noindent where $\langle \nabla_{\xi} \hat{\mc{H}}(dt, d \hat{\xi}(t)), \beta(t) \rangle := \lim_{a \rightarrow 0^+} \frac{\hat{\mc{H}}(\hat{\xi} + a \beta)(dt, d(\hat{\xi} + a \beta)) -\hat{\mc{H}}(\hat{\xi})(dt, d\hat{\xi})  }{a}$.

\end{corollary}

\begin{proof}
This is a direct consequence of Proposition~\ref{thm: necessary}, the comments following the proposition and the following calculation:

\[
\begin{array}{llll}
 \langle \nabla_{\xi} \hat{\mc{H}}(dt, d \hat{\xi}(t)), \beta(t) \rangle &=& \lim_{a \rightarrow 0^+} \frac{\hat{\mc{H}}(\hat{\xi} + a \beta)(dt, d(\hat{\xi} + a \beta)) -\hat{\mc{H}}(\hat{\xi})(dt, d\hat{\xi})  }{a}\\[\smallskipamount]
 &=& \lim_{a \rightarrow 0^+} \frac{\hat{H}_1(\hat{\xi} + a \beta)dt -\hat{H}_1(\hat{\xi})dt}{a} \\[\smallskipamount]
 &&+ \lim_{a \rightarrow 0^+} \frac{\hat{H}_2(\hat{\xi} + a \beta)d\hat{\xi}(t) -\hat{H}_2(\hat{\xi})d\hat{\xi}(t)}{a} \\[\smallskipamount]
 && + \lim_{a \rightarrow 0^+} \frac{\hat{H}_2(\hat{\xi} + a \beta)d(a\beta)(t)}{a} \\[\smallskipamount]
 &&= \beta(t)\{ \hat{\lambda}(t) \frac{\partial \hat{g}_1}{\partial \xi}(t) dt + \hat{\lambda}(t) \frac{\partial \hat{g}_2}{\partial \xi}(t) d \hat{\xi}(t) \}  \\[\smallskipamount]
 &&+ \hat{H}_2(t)d \beta(t).
\end{array}
\]

\end{proof}

We analyse the inequality from the second item of Corollary~\ref{corollary: rewritten} more closely, i.e. we consider:

\[
 \begin{array}{lll}
E[\int_0^T \hat{H}_2(s) d \beta(s)] +  E[\int_0^T \beta(s) \{\hat{\lambda}(s) \frac{\partial \hat{g}_2}{\partial \xi}(s) d \hat{\xi}(s) + \hat{\lambda}(s) \frac{\partial \hat{g}_1}{\partial \xi}(s) ds  \} ]  \leq 0. \\[\smallskipamount]
\end{array}
\]

Since this is true for all $\beta \in \mc{V}(\hat{\xi})$, it is in particular true for $\beta(s):= \boldsymbol{1}_{[t,T]}(s)\alpha(\omega) \hat{\xi}(s)$ and $\beta(s):= -\boldsymbol{1}_{[t,T]}(s) \alpha(\omega) \hat{\xi}(s)$, where $\hat{\xi}(s)$ is as in Corollary~\ref{corollary: rewritten} and $\alpha = \alpha(\omega)$ is a bounded $\mathcal{F}_t$-measurable random variable. By combining the two, we see that

\[
 \begin{array}{lll}
0 \leq  E[\int_t^T \hat{H}_2(s) d \hat{\xi}(s)\alpha] +  E[\int_t^T \hat{\xi}(s) \{\hat{\lambda}(s) \frac{\partial \hat{g}_2}{\partial \xi}(s) d \hat{\xi}(s) + \hat{\lambda}(s) \frac{\partial \hat{g}_1}{\partial \xi}(s) ds  \} \alpha] \leq 0. \\[\smallskipamount]
\end{array}
\]

Hence,

\begin{equation}
\label{eq: likhet}
\begin{array}{lll}
E[\int_t^T \hat{H}_2(s) d \hat{\xi}(s)\alpha] +  E[\int_t^T \hat{\xi}(s) \{\hat{\lambda}(s) \frac{\partial \hat{g}_2}{\partial \xi}(s) d \hat{\xi}(s) + \hat{\lambda}(s) \frac{\partial \hat{g}_1}{\partial \xi}(s) ds  \} \alpha]  = 0. \\[\smallskipamount]
\end{array}
\end{equation}

By differentiating the equality~\eqref{eq: likhet} with respect to $t$, we see that

\[
E [\big( \{ \hat{p}(t)\hat{\theta}(t) + \hat{\lambda}(t)\hat{g}_2(t) \} d \hat{\xi}(t) + \hat{\xi}(t)\hat{\lambda}(t)\{ \frac{\partial \hat{g}_2}{\partial \xi}(t) d \hat{\xi}(t) + \frac{\partial \hat{g}_1}{\partial \xi}(t) dt\} \big) \alpha] =0
\]

\noindent  for almost all $t$. Since this holds for all bounded $\mathcal{F}_t$-measurable random variables $\alpha$, we conclude that
\[
\{ \hat{p}(t)\hat{\theta}(t) + \hat{\lambda}(t)\hat{g}_2(t)\} d \hat{\xi}(t) + \hat{\xi}(t)\hat{\lambda}(t)\{ \frac{\partial \hat{g}_2}{\partial \xi}(t) d \hat{\xi}(t) + \frac{\partial \hat{g}_1}{\partial \xi}(t) dt\}  =0.
\]

This is related to the first order condition for $\xi = \hat{\xi}$ to be optimal in \eqref{eq: notat_(10)} (the condition of the sufficient maximum principle Theorem~\ref{eq: Suff_Max_Princ}).
More precisely, this is what we get if we differentiate the function
\begin{equation}
a \mapsto \hat{\lambda}(t)g_1(\hat{\xi}+a\beta)dt+\{ \hat{p}(t)\hat{\theta}(t)+\hat{\lambda}(t)g_2(\hat{\xi}+a\beta)(t)\}d(\hat{\xi}(t)+a\beta(t))
\end{equation}
with respect to $a$ at $a=0$, set this derivative equal to $0$ and then evaluate the result at $\beta=\hat{\xi}$.

\section{Applications}
\label{sec: applications}

\begin{example}
 \label{ex: 7.1}
 Suppose we have a cash flow $X(t) = X^{(\xi)}(t)$ of the form:

 \begin{equation}
  \label{eq: 7.1}
  \begin{array}{llll}
   dX(t) &=& X(t)[b_0(t)dt + \sigma_0(t) dB(t)] - X(t^-) d\xi(t); \mbox{ } t \in [0,T] \\[\smallskipamount]
   X(0) &=& x > 0.
  \end{array}
 \end{equation}

Here $d\xi(t)$ represents the relative consumption rate from $X(t)$ at time $t$.\\

The singular recursive utility process $Y(t) = Y^{(\xi)}(t)$ of the relative consumption rate $\xi(t)$ is assumed to have the form

\begin{equation}
 \label{eq: 7.2}
 \begin{array}{lll}
  dY(t) &=& -\alpha(t)Y(t) d \xi(t) + Z(t) dB(t); \mbox{ } t \in [0,T] \\[\smallskipamount]
  Y(T) &=& h(X(T)).
 \end{array}
\end{equation}

 We want to find $\xi^* \in \mc{A}$ such that

 \begin{equation}
  \label{eq: 7.3}
  Y^{(\xi^*)}(0) = \sup_{\xi \in \mc{A}} Y^{(\xi)}(0).
 \end{equation}
\end{example}

 We apply the results of Section~\ref{sec: max_sing_rec_utility} to study this problem:

 The Hamiltonian~\eqref{eq: notat_(6)} gets the form
 \begin{equation}
  \label{eq: 7.4}
  \begin{array}{lll}
   \mc{H}(t,x,y,\xi,p,q,\lambda)(dt,d\xi) = (x b_0(t) p + x \sigma_0(t)q)dt + (-x p + \lambda \alpha(t)y) d\xi(t)
  \end{array}
 \end{equation}

 The adjoint equations ~\eqref{eq: notat_(7)}-\eqref{eq: notat_(8)} become

 \begin{equation}
  \label{eq: 7.5}
  \begin{array}{llll}
   d\lambda(t) &=& \lambda(t) \alpha(t) d\xi(t); \mbox{ } t \in [0,T] \\[\smallskipamount]
   \lambda(0) &=& 1
  \end{array}
 \end{equation}

 \begin{equation}
  \label{eq: 7.6}
  \begin{array}{llll}
   dp(t) &=& -(b_0(t)p(t) + \sigma_0(t)q(t))dt - \lambda(t) \alpha(t) d\xi(t) + q(t) dB(t); \mbox{ } t \in [0,T] \\[\smallskipamount]
   p(T) &=& \lambda(T) h'(X(T))
  \end{array}
 \end{equation}

 \noindent and the variational inequality (Skorohod reflection problem) ~\eqref{eq: 27} reduces to

 \begin{equation}
  \label{eq: 7.7}
  \begin{array}{lll}
   -\hat{p}(t) \hat{X}(t) + \hat{\lambda}(t)\alpha(t) \hat{Y}(t) &\leq& 0 \mbox{ for all } t \in [0,T], \\[\smallskipamount]
   \{-\hat{p}(t) \hat{X}(t) + \hat{\lambda}(t)\alpha(t) \hat{Y}(t)\} d\hat{\xi}(t) &=& 0 \mbox{ for all } t \in [0,T].
  \end{array}
 \end{equation}

 Thus, we arrive at the following conclusion:

 \begin{theorem}
  \label{thm: 7}
  Suppose the Skorohod reflection problem~\eqref{eq: 7.7} has a solution \\
  $(\hat{X}(t), \hat{Y}(t), \hat{\lambda}(t), \hat{p}(t), \hat{\xi}(t))$. Then $\hat{\xi}(t)$ is an optimal relative consumption rate for the singular recursive utility problem~\eqref{eq: 7.3}.
 \end{theorem}

For more on Skorohod problems and conditions guaranteeing the existence of solutions to such problems, see {\O}ksendal and Sulem (2007), chapter 5.2.

\end{document}